\documentclass[11pt, reqno]{amsart}
\usepackage[utf8]{inputenc}
\usepackage{amsmath}
\usepackage{amsthm}
\usepackage{amsfonts}
\usepackage{amssymb}
\usepackage{blindtext,amsmath,amsthm,amsfonts,amssymb, textcomp, mathrsfs, mathtools}

\usepackage[top=32mm,bottom=32mm,left=30mm,right=30mm]{geometry}
\newtheorem{theorem}{Theorem}[section]
\newtheorem{conjecture}{Conjecture}
\newtheorem{corollary}{Corollary}

\newtheorem*{theorem*}{Theorem}
\newtheorem*{remark*}{Remark}
\newtheorem*{problem*}{Problem}
\newtheorem*{conjecture*}{Conjecture}
\newtheorem{lemma}[theorem]{Lemma}

\newcommand*{\QED}{\hfill\ensuremath{\square}}

\begin{document}
\title[Derivatives of $L$-series and generalized Stieltjes constants]{Special values of derivatives of $L$-series and generalized Stieltjes constants}

\author[M. Ram Murty and Siddhi Pathak]{M. Ram Murty\textsuperscript{1} and Siddhi Pathak}

\address{Department of Mathematics and Statistics, Queen's University, Kingston, Canada, ON K7L 3N6.}
\email{murty@mast.queensu.ca}
\address{Department of Mathematics and Statistics, Queen's University, Kingston, Canada, ON K7L 3N6.}
\email{siddhi@mast.queensu.ca}

\subjclass[2010]{11M41}

\keywords{derivative of $L$-series, generalized Stieltjes constants, Rohrlich's conjecture}

\footnotetext[1]{Research of the first author was supported by an NSERC Discovery grant.}

\begin{abstract}
The connection between derivatives of $L(s,f)$ for periodic arithmetical functions $f$ at $s=1$ and generalized Stieltjes constants has been noted earlier. In this paper, we utilize this link to throw light on the arithmetic nature of $L'(1,f)$ and certain Stieltjes constants. In particular, if $p$ is an odd prime greater than $7$, then we deduce the transcendence of at least $(p-7)/2$ of the generalized Stieltjes constants, $\{ \gamma_1(a,p) : 1 \leq a < p \}$, conditional on a conjecture of S. Gun, M. Ram Murty and P. Rath \cite{gmr}.

\end{abstract}

\maketitle
\begin{center}
{\sl  To Professor Robert Tijdeman on the occasion of his 75th birthday}
  \end{center}
\section{\bf Introduction}
\bigskip

The Riemann zeta function $\zeta(s)$ plays a crucial role in mathematics. The Laurent series expansion of $\zeta(s)$ around $s=1$ (see \cite{chowla}) can be written as
\begin{equation*}
    \zeta(s) = \frac{1}{s-1} + \gamma + \sum_{k=1}^{\infty} \frac{{(-1)}^k}{k!} \, \gamma_k \, {(s-1)}^k,
\end{equation*}
where
\begin{equation*}
    \gamma_k := \lim_{N \rightarrow \infty} \bigg\{ \sum_{n=1}^N \bigg( \frac{\log^k n}{n} \bigg) - \frac{\log^{(k+1)} N}{k+1} \bigg\}
\end{equation*}
are called Stieltjes constants and $\gamma$ is the well known Euler-Mascheroni constant. Even though these constants are important ingredients of the theory of the Riemann zeta function and appear in many contexts, it is unknown whether they are rational or irrational although they are expected to be transcendental. As a generalization of this question to arithmetic progressions, Knopfmacher \cite{knopf} defined 
\begin{equation*}
    \gamma_k(a,q) := \lim_{x \rightarrow \infty} \bigg\{ \sum_{\substack{n \leq x, \\ n \equiv a \bmod q}} \bigg( \frac{\log^k n}{n} \bigg) - \frac{\log^{(k+1)}x}{q(k+1)} \bigg\},
\end{equation*}
for natural numbers $a$ and $q$. The case $k=0$ was studied earlier by D. H. Lehmer \cite{lehmer} in 1975. We refer to these constants as \textit{generalized Stieltjes constants}.
\par
The motivation for studying these constants emanates from our desire
to understand special values of $L$-series.  More precisely,
when $f$ is an arithmetical function, with period $q$, the
Dirichlet series
$$L(s,f):= \sum_{n=1}^\infty {f(n) \over n^s}, $$
has been the focus of intense study (see for example the
survey article by Tijdeman \cite{tijdeman} as well as \cite{tmp}, \cite{gmr}
and \cite{ram-saradha}).  However, these papers studied the special
value $L(1,f)$ whenever it is defined.  Interestingly, this special
value can be studied using Baker's theory of linear forms in logarithms.
In this paper, our focus will be on the derivative $L'(1,f)$.
This problem has received scant attention.  For example, there is the
curious result of Murty and Murty \cite{murty-murty} which states
that if there is some squarefree $D>0$ and $\chi_D$ is the quadratic
character attached to ${\mathbb Q}(\sqrt{-D})$ is such that  $L'(1,\chi_D)=0$,
then, $e^\gamma$ is transcendental.  An analogous question of non-vanishing
seems to occur in other contexts as well (see for example, \cite{scourfield}).
These are not unrelated to the Euler-Kronecker constants
studied in \cite{ihara-murty} and \cite{mourtada-murty}.
\par

Many arithmetic properties and computational aspects of these constants have been studied in \cite{dilcher} but the only known result about their transcendental nature is a theorem due to M. Ram Murty and N. Saradha \cite[Theorem 1]{ram-saradha}, where they tackle the case $k=0$. In this paper, we concentrate on the arithmetic nature of these constants when $k=1$.

The nature of values of the gamma function at rational arguments and relations among them has been the subject of research for a long time. In light of this, a conjecture put forth by S. Gun, M. Ram Murty and P. Rath \cite{gmr} will be useful towards a partial solution to our question. The conjecture is the following.
\begin{conjecture}\label{conj}
For any positive integer $q > 2$, let $\overline{V_{\Gamma}(q)}$ be the $\overline{\mathbb{Q}}$-vector space spanned by the real numbers
\begin{equation*}
    \log\Gamma \bigg(\frac{a}{q}\bigg), \hspace{1mm} 1 \leq a \leq q, \hspace{1mm} (a,q) = 1.
\end{equation*}
Then the dimension of $\overline{V_{\Gamma}(q)}$ is $\phi(q)$.
\end{conjecture}
This conjecture was inspired by a conjecture of Rohrlich (see \cite{waldschmidt}) regarding the possible relations among the special values of the $\Gamma$-function.  
We note that the conjecture is equivalent to the numbers $\{ \log\Gamma(a/q) : 1 \leq a \leq q, (a,q) = 1\}$ being $\overline{\mathbb{Q}}$-linearly independent for $q>2$. This is a major unsolved problem in number theory and is believed to be outside the scope of current mathematical tools. 

For a natural number $q$, a function $f$ defined on the integers which is periodic with period $q$ is said to be {\sl odd} if $f(q-n) = - f(n)$ for all natural numbers $n$.
It is said to be {\sl even} if $f(q-n)=f(n)$ for all natural numbers $n$.
As noted earlier, the $L$-series attached to $f$ is defined as
\begin{equation*}
    L(s,f) := \sum_{n=1}^{\infty} \frac{f(n)}{n^s},
\end{equation*}
for $\Re(s)>1$. Using the theory of the Hurwitz zeta function, $L(s,f)$ can be extended to an entire function as long as $\sum_{a=1}^q f(a) = 0$. Given a function $f$ which is periodic mod $q$, we define the Fourier transform of $f$ as
\begin{equation}\label{fourier-def}
\hat{f}(b) := \frac{1}{q} \sum_{a=1}^q f(a) \zeta_q^{-ab}, 
\end{equation}
where $\zeta_q = e^{2 \pi i/q}$.
This can be inverted using the identity
\begin{equation*}\label{fourier-inversion}
f(n) = \sum_{b=1}^q \hat{f}(b) \zeta_q^{bn}.
\end{equation*}
Thus, the condition for convergence of $L(1,f)$, i.e, $\sum_{a=1}^q f(a) = 0$ can be interpreted as $\hat{f}(q) = 0$. An arithmetical function periodic with period $q$ is said to be of \textit{Dirichlet type} if 
\begin{equation*}
    f(n) = 0, \hspace{1mm} \text{whenever} \quad (n,q) > 1.
\end{equation*}
Another important notion is that of the linear independence of arithmetical functions. A set of arithmetical functions  $\{f_1,f_2, \cdots, f_m\}$ is said to be linearly independent over $\overline{\mathbb{Q}}$ if 
\begin{equation*}
    \sum_{j=1}^m \alpha_j f_j = 0, \text{ with } \alpha_j \in \overline{\mathbb{Q}} \Rightarrow \alpha_j = 0 \hspace{1mm} \text{ for all } 1 \leq j \leq m.
\end{equation*}

Again, using the theory of the Hurwitz zeta function, one can derive
formulas (as we will see below) for $L'(1,f)$ in terms of the Stieltjes
constants.

With this discussion in place, we state the main theorems of this paper.
\begin{theorem}\label{main-theorem-1}
Let $p$ be a prime greater than $7$. Define
\begin{equation*}
        \mathfrak{F}_p := \bigg\{ f : \mathbb{Z} \rightarrow \overline{\mathbb{Q}} \hspace{1mm} | 
        \, \hspace{1mm} f \text{ is periodic with period } p, \, f \text{ is odd}, \, \widehat{f}(p) = 0, \, L(1,f) \neq 0 \bigg\}.
\end{equation*}
For $r > 2$, let $f_1, f_2, \cdots, f_r $ be $\overline{\mathbb{Q}}$-linearly independent elements of $\mathfrak{F}_p$. Then Conjecture \ref{conj} implies that at most three of the numbers
\begin{equation*}
    \bigg\{ L'(1,f_j) = - \sum_{a=1}^p f_j(a) \gamma_1(a,p) \hspace{1mm} | \hspace{1mm} 1 \leq j \leq r \bigg\}
\end{equation*}
are algebraic.
\end{theorem}

We also handle the case for arithmetical functions periodic with period $p$ and $L(1,f)=0$ in the theorem below.
\begin{theorem}\label{main-theorem-2}
Let $p$ be a prime number greater than 5. Define
\begin{equation*}
    \mathfrak{G}_p := \bigg\{ f: \mathbb{Z} \rightarrow \overline{\mathbb{Q}} \hspace{1mm} | \hspace{1mm} 
    f \text{ is periodic with period } p, \, f \text{ is odd}, \, \widehat{f}(p) = 0, \, L(1,f) = 0 \bigg\}.
\end{equation*}
For $r \geq 2$, let $f_1, \cdots, f_r$ be $\overline{\mathbb{Q}}$-linearly independent elements of $\mathfrak{G}_p$. Then under Conjecture \ref{conj}, we conclude that at most one of the numbers 
\begin{equation*}
    \bigg\{ L'(1, f_j ) = - \sum_{a=1}^p f_j(a) \gamma_1(a,p) \hspace{1mm} | \hspace{1mm} 1 \leq j \leq r \bigg\}
\end{equation*}
is algebraic.
\end{theorem}

\begin{remark*}
Both the above theorems also hold when the period of the functions under consideration is a composite number $q$, provided that the Fourier transforms of the functions are of Dirichlet type. This restriction comes from the nature of Conjecture \ref{conj}. Indeed, we will prove that Theorem \ref{main-theorem-1} holds for the general set of functions
\begin{equation*}
    \begin{split}
        \mathfrak{F}_q := \bigg\{ f: \mathbb{Z} \rightarrow \overline{\mathbb{Q}} \hspace{1mm} | & \hspace{1mm} 
        f \text{ is periodic with period } q, \, f \text{ is odd}, \\
        & \widehat{f} \text{ is of Dirichlet type}, \, L(1,f) \neq 0 \bigg\},
    \end{split}
\end{equation*}
and Theorem \ref{main-theorem-2} holds for
\begin{equation*}
    \begin{split}
        \mathfrak{G}_q := \bigg\{ f: \mathbb{Z} \rightarrow \overline{\mathbb{Q}} \hspace{1mm} | & \hspace{1mm} 
        f \text{ is periodic with period } q, \, f \text{ is odd}, \\
        & \widehat{f} \text{ is of Dirichlet type}, \, L(1,f) = 0 \bigg\},
    \end{split}
\end{equation*}
where $q$ is not necessarily prime.
\end{remark*}

We note that the above defined set $\mathfrak{F}_q$ is non-empty since odd primitive Dirichlet characters modulo $q$ are in $\mathfrak{F}_q$. We will see this in the course of proof of the following corollary.
\begin{corollary}\label{coro-1}
Let $q$ be a natural number greater than $7$. Then assuming Conjecture \ref{conj}, we deduce that at most three of the following numbers are algebraic:
\begin{equation*}
    \bigg\{ L'(1,\chi) = - \sum_{a=1}^q \chi(a) \gamma_1(a,q) \hspace{1mm}| \hspace{1mm} \chi \text{ is an odd primitive Dirichlet character mod } q \bigg\}.
\end{equation*}
\end{corollary}

Applying Theorem \ref{main-theorem-1} to the scenario when $q = p$, an odd prime greater than $7$ and
\begin{equation*}
    f_j(n) :=
    \begin{cases}
    1 & \text{ if } n \equiv j \bmod p, \\
    -1 & \text{ if } n \equiv -j \bmod p, \\
    0 & \text{ otherwise,}
    \end{cases}
\end{equation*}
for $1 \leq j \leq (p-1)/2$, we infer the following:
\begin{corollary}\label{coro-2}
For an odd prime $p$ greater than $7$, Conjecture \ref{conj} implies that at least $(p-7)/2$ of the numbers
\begin{equation*}
    \big\{ \gamma_1(a,p) \hspace{1mm} | \hspace{1mm} 1 \leq a \leq p-1 \big\}
\end{equation*}
are transcendental.
\end{corollary}

\section{\bf Preliminaries}
\bigskip

The aim of this section is to introduce notation and some fundamental results that will be used in the later part of the paper. Let $q$ be a fixed positive integer. Consider $f : \mathbb{N} \to \overline{\mathbb{Q}}$, periodic with period $q$. Define
\begin{equation*}
L(s,f) = \sum_{n=1}^{\infty} \frac{f(n)}{n^s}.
\end{equation*}
Observe that $L(s,f)$ converges absolutely for $\Re(s) > 1$. Since $f$ is periodic,
\begin{align}
L(s,f) & = \sum_{a=1}^q f(a)\sum_{k=0}^{\infty} \frac{1}{{(a + kq)}^s} \nonumber \\
& = \frac{1}{q^s} \sum_{a=1}^q f(a) \zeta(s, a/q), \label{l-zeta}
\end{align}
where $\zeta(s,x)$ is the Hurwitz zeta function. For $\Re(s) > 1$ and $0 < x \leq 1$, the Hurwitz zeta function is defined as
\begin{equation*}
\zeta(s,x) = \sum_{n=0}^{\infty} \frac{1}{{(n+x)}^s}.
\end{equation*}
In 1882, Hurwitz \cite[Chapter 12, Section 5]{apostol} proved that $\zeta(s,x)$ has an analytic continuation to the entire complex plane except for a simple pole at $s=1$ with residue $1$. Using this, we  conclude that $L(s,f)$ can be extended analytically to the entire complex plane except for a simple pole at $ s=1 $ with residue $ \frac {1} {q} \sum_{a=1}^q f(a)$. Thus,
it is easy to deduce that
$\sum_{n=1}^\infty \frac {f(n)} {n}$ converges whenever  $\sum_{a=1}^q f(a) = 0$. Hence, $f(q) = \hat{f}(q) = 0$ implies that both $L(s,f)$ as well as $L(s, \hat{f})$ are entire.\\

Before proceeding, we prove a few lemmas for arithmetical functions periodic with period $q$.

\begin{lemma}\label{lemma-1}
Let $f$ be an arithmetical function periodic with period $q$. Then,
\begin{equation*}
    L(1-s,f) = 2 \, \Gamma(s) {\bigg(\frac{q}{2 \pi}\bigg)}^s \cos\bigg( \frac{s\pi}{2}\bigg) L(s,\hat{f}),
\end{equation*}
when $f$ is even (i.e., f(q-n) = f(n) for all $n$) and
\begin{equation*}
    L(1-s,f) = 2\,i \, \Gamma(s) {\bigg(\frac{q}{2 \pi}\bigg)}^s \sin\bigg( \frac{s\pi}{2}\bigg) L(s,\hat{f}),
\end{equation*}
when $f$ is odd (i.e, f(q-n) = -f(n) for all $n$).
\end{lemma}
\begin{proof}
We refer the reader to \cite[Chapter XIV, Theorem 2.1]{lang}.
\end{proof}

In analogy with the notation of generalized Bernoulli numbers associated to Dirichlet characters, we define
\begin{equation*}
    B_{1,f} := \sum_{a=1}^q a \, f(a),
\end{equation*}
where $f$ is an odd arithmetical function periodic with period $q$. We make another important observation.
\begin{lemma}\label{lemma-2}
For any arithmetical function $f$ periodic with period $q$,
\begin{equation*}
    L'(0,f) = \frac{\log q}{q} \, B_{1,f} + \sum_{b=1}^q f(b) \log\Gamma \bigg( \frac{b}{q} \bigg).
\end{equation*}
\end{lemma}
\begin{proof}
 By differentiating \eqref{l-zeta} with respect to $s$, we have
 \begin{equation*}
     L'(s,f) =  \frac{- \log q}{q^s} \bigg[ \sum_{a=1}^q f(a) \, \zeta \bigg(s, \frac{a}{q} \bigg) \bigg] + \bigg[ \frac{1}{q^s} \sum_{a=1}^q f(a) \, \zeta' \bigg(s, \frac{a}{q} \bigg) \bigg].
 \end{equation*}
Substituting $s=0$, we have
 \begin{equation*}
     L'(0,f) = - \log q \bigg[ \sum_{a=1}^q f(a)\,\zeta \bigg(0, \frac{a}{q} \bigg)  \bigg] + \bigg[ \sum_{a=1}^q f(a) \zeta' \bigg(0, \frac{a}{q} \bigg)\bigg].
 \end{equation*}
 
The values of the Hurwitz zeta function and its derivative at $s=0$ are given by
\begin{equation*}
    \zeta(0,x) = 1 + \zeta(0) - x, \hspace{1mm} \zeta'(0,x) = \log\Gamma(x) + \zeta'(0), 
\end{equation*}
 where $\zeta(s)$ is the Riemann zeta function. A proof of the above fact can be found in \cite{deninger}. Substituting these values in the expression obtained earlier, we get
 \begin{align*}
     L'(0,f) & = - \log q \, \big(1 + \zeta(0)\big) \bigg[ \sum_{a=1}^q  f(a) \bigg] + \frac{\log q}{q} \sum_{a=1}^q f(a)\, a \\
     & + \zeta'(0) \bigg[ \sum_{a=1}^q f(a) \bigg] + \sum_{a=1}^q f(a) \log\Gamma\bigg( \frac{a}{q} \bigg) \\
     & = \frac{\log q}{q} \sum_{a=1}^q f(a) \, a +  \sum_{a=1}^q f(a) \log\Gamma\bigg( \frac{a}{q} \bigg),
 \end{align*}
since $\sum_{a=1}^q f(a) = 0$. This proves the lemma. 
\end{proof}

The functional equation obtained in Lemma \ref{lemma-1} gives an expression for the value of $L'(1,f)$ when $f$ is an odd periodic function as below.
\begin{lemma} \label{lemma-3}
Let $f$ be an odd periodic arithmetical function with period $q$ satisfying $f(q) = \hat{f}(q) = 0$. Then,
\begin{equation*}
    L'(1,f) = \frac{i \pi}{q} \bigg\{ \bigg( \bigg( 1 + \frac{1}{q} \bigg) \log q - \log 2 \pi - \gamma \bigg) B_{1,\hat{f}} + \sum_{b=1}^q \hat{f}(b) \log \Gamma \bigg( \frac{b}{q} \bigg) \bigg\},
\end{equation*}
where $B_{1,g} := \sum_{a=1}^q a g(a) $ for any odd arithmetical function periodic with period $q$.
\end{lemma}
\begin{proof}
Note that if $f$ is an odd periodic arithmetical function, then so is $\hat{f}$. Thus, differentiating the functional equation for $L(s,\hat{f})$ from Lemma \ref{lemma-1} gives
\begin{multline*}
    - L'(1-s,\hat{f}) = 2 \, \Gamma'(s) {\bigg(\frac{q}{2 \pi}\bigg)}^s \sin\bigg( \frac{s\pi}{2}\bigg) L(s,f) \\
    + 2 \, \Gamma(s) {\bigg(\frac{q}{2 \pi}\bigg)}^s \log \bigg( \frac{q}{2 \pi} \bigg) \sin\bigg( \frac{s\pi}{2}\bigg) L(s,f) \\
    + 2 \, \Gamma(s) {\bigg(\frac{q}{2 \pi}\bigg)}^s \, \frac{\pi}{2} \cos\bigg( \frac{s\pi}{2}\bigg) L(s,f) \\
    +  2 \, \Gamma(s) {\bigg(\frac{q}{2 \pi}\bigg)}^s \sin\bigg( \frac{s\pi}{2}\bigg) L'(s,f).
\end{multline*}

Since $f(q) = \hat{f}(q) = 0$, both $L(s,f)$ and $L(s,\hat{f})$ are entire. Taking limit as $s$ tends to 1 in the above expression, we have
\begin{align*}
    - L'(0,\hat{f}) & = 2 \, i \, \Gamma(1) \frac{q}{2\pi} \sin \bigg( \frac{\pi}{2} \bigg) \bigg\{ \bigg(\frac{\Gamma'}{\Gamma}(1) + \log \bigg( \frac{q}{2 \pi} \bigg) \bigg) L(1,f) + L'(1,f) \bigg\} \\
    & = \frac{i q}{\pi} \bigg\{ L'(1,f) + L(1,f) \bigg( \log\bigg( \frac{q}{2\pi} \bigg) - \gamma \bigg) \bigg\},
\end{align*}
as $\Gamma'(1)/ \Gamma (1) = - \gamma$. By rearrangement, we get
\begin{equation*}
    L'(1,f) = \frac{i \pi}{q} L'(0,\hat{f}) - \bigg( \log \bigg(\frac{q}{2 \pi} \bigg) - \gamma \bigg) L(1,f).
\end{equation*}

The value $L(1,f)$ for periodic arithmetical functions is well-understood (for example, see \cite[Theorem 3.1]{tmp}). In particular, when $f$ is odd,
\begin{equation}\label{L(1,f)-odd}
    L(1,f) =  \frac{- i \pi}{q} \sum_{a=1}^q \hat{f}(a) \, a =  \frac{- i \pi}{q} \, B_{1,\hat{f}},
\end{equation}
where $\hat{f}$ denotes the Fourier transform of $f$. This evaluation, together with Lemma \ref{lemma-2} gives
\begin{align*}
    L'(1,f) & = \frac{i \pi}{q} \bigg\{ \frac{\log q}{q} B_{1,\hat{f}} +  \sum_{b=1}^q \hat{f}(b) \log\Gamma \bigg( \frac{b}{q} \bigg)  + \bigg( \log \bigg(\frac{q}{2 \pi} \bigg) - \gamma \bigg) B_{1,\hat{f}} \bigg\},\\
    & = \frac{i \pi}{q} \bigg\{ \bigg( \frac{\log q}{q} + \log q - \log 2 \pi - \gamma \bigg) B_{1,\hat{f}} + \sum_{b=1}^q \hat{f}(b) \log\Gamma \bigg( \frac{b}{q} \bigg)  \bigg\},
\end{align*}
from which the lemma is immediate.
\end{proof}

A useful connection between value of derivatives of $L$-functions, attached to periodic functions at $s=1$ and generalized Stieltjes constants is illustrated in the following lemma. We include its proof for the sake of exposition (see \cite[Proposition 3.2]{knopf}).
\begin{lemma}\label{identity-lemma}
For an arithmetical function $f$ which is periodic with period $q$ and satisfies $\hat{f}(q) = 0$, we have
\begin{equation*}
    {L}^{(k)}(1,f) = {(-1)}^k \sum_{a=1}^q f(a) \gamma_k(a,q),
\end{equation*}
where $\gamma_k(a,q)$ are generalized Stieltjes constants as defined earlier. 
\end{lemma}
\begin{proof}
In view of brevity, let
\begin{equation*}
    H_k(x,a,q) := \sum_{\substack{n \leq x, \\ n \equiv a \bmod q}} \frac{\log^k n}{n},
\end{equation*}
for any positive real number $x$. Observe that
\begin{align*}
    \sum_{n \leq x} f(n) \, \frac{\log^k n}{n} & = \sum_{a=1}^q f(a) H_k(x,a,q) \\
    & = \sum_{a=1}^q f(a) \bigg( H_k(x,a,q) - \frac{\log^{k+1 x}}{q(k+1)}\bigg),
\end{align*}
since $q \, \hat{f}(q) = \sum_{a=1}^q f(a) = 0$. Taking limit as $x$ tends to infinity on both sides gives the result.
\end{proof}

As mentioned earlier, the special value $L(1,f)$ has been extensively studied and is important in the context of our theorems. The following result of Baker, Birch and Wirsing \cite{bbw} will be particularly useful.
\begin{theorem}\label{bbw-thm}
If $f$ is a non-vanishing function defined on the integers with algebraic values and period $q$ such that (i) $f(n) = 0$ whenever $1 < (n,q) < q$ and (ii) the $q^{th}$ cyclotomic polynomial $\Phi_q$ is irreducible over $\mathbb{Q} (f(1), f(2), \cdots , f(q))$, then
\begin{equation*}
\sum_{n=1}^\infty \frac{f(n)} {n} \neq 0.
\end{equation*} 

\end{theorem}

We also observe that if $f_1, \cdots, f_r$ are arithmetical functions periodic with period $q$, then 
\begin{equation}\label{lin-ind-f-hat}
    \sum_{j=1}^r \alpha_j f_j = 0 \iff \sum_{j=1}^r \alpha_j \widehat{f_j} = 0,
\end{equation}
for any complex numbers $\alpha_j$, $1 \leq j \leq r$. This is immediate from the fact that the Fourier transform is a linear automorphism of the $\mathbb{C}$-vector space of arithmetical functions periodic with period $q$.

\section{\bf Proofs of Results}
\bigskip

\subsection{Proof of Theorem \ref{main-theorem-1}} 
For convenience of notation, let 
\begin{equation*}
    \mathscr{C} := \bigg( 1 + \frac{1}{q} \bigg) \log q - \log 2 \pi - \gamma.
\end{equation*}
Thus, Lemma \ref{lemma-3} gives
\begin{equation*}
    L'(1,f_j) = \frac{i \pi}{q} \bigg\{ \mathscr{C} B_{1,\hat{f_j}} + \sum_{b=1}^q \hat{f_j}(b) \log\Gamma \bigg( \frac{b}{q} \bigg) \bigg\},
\end{equation*}
for all $1 \leq j \leq r$. By \eqref{L(1,f)-odd}, the hypothesis $L(1,f_j) \neq 0$ implies that $B_{1,\hat{f_j}} \neq 0$. \\

For $1 \leq k < l \leq r$, define
\begin{equation*}
    d_{k,l} := B_{1,\hat{f_l}} L'(1,f_k) - B_{1,\hat{f_k}} L'(1,f_l). 
\end{equation*}

We claim that $d_{k,l} \neq 0$. Indeed, if $d_{k,l} = 0$, then we get that
\begin{align*}
    0 & = B_{1,\hat{f_l}} L'(1,f_k) - B_{1,\hat{f_k}} L'(1,f_l) \\
    & = \frac{i \pi}{q} \bigg\{ \mathscr{C} \bigg( B_{1,\hat{f_l}} B_{1,\hat{f_k}} - B_{1,\hat{f_k}} B_{1, \hat{f_l}} \bigg) + \sum_{b=1}^q \bigg[ B_{1,\hat{f_l}} \hat{f_k}(b) - B_{1, \hat{f_K}} \hat{f_l}(b) \bigg] \log\Gamma\bigg( \frac{b}{q}\bigg) \bigg\} \\
    & =  \sum_{b=1}^q \bigg[ B_{1,\hat{f_l}} \hat{f_k}(b) - B_{1, \hat{f_K}} \hat{f_l}(b) \bigg] \log\Gamma\bigg( \frac{b}{q}\bigg),
\end{align*}
which is a $\overline{\mathbb{Q}}$-linear relation among the values of the
log gamma function as $B_{1,\hat{f_j}} \in \overline{\mathbb{Q}}$ for all $1 \leq j \leq r$. Therefore, Conjecture \ref{conj} gives that
\begin{equation*}
    B_{1,\hat{f_l}} \hat{f_k} - B_{1,\hat{f_k}} \hat{f_l} = 0
\end{equation*}
on all natural numbers. This implies $\overline{\mathbb{Q}}$-linear dependence of $\hat{f_k}$ and $\hat{f_l}$ and thus, contradicts the $\overline{\mathbb{Q}}$-linearly independence of $f_k$ and $f_l$ by \eqref{lin-ind-f-hat}. Hence, $d_{k,l}$ is not zero. 

We now consider the ratio $d_{k,l}/d_{u,v}$ for $1 \leq k,u < l,v \leq r$ and $(k,l) \neq (u,v)$. If this ratio is algebraic, i.e.,
\begin{equation*}
    \frac{d_{k,l}}{d_{u,v}} = \eta \hspace{1mm} \in \overline{\mathbb{Q}},
\end{equation*}
then we are led to argue that
\begin{align*}
    0 & = d_{k,l} - \eta d_{u,v} \\
    & = \sum_{b=1}^q \bigg[ B_{1,\hat{f_l}} \hat{f_k}(b) - B_{1,\hat{f_k}} \hat{f_l}(b) - \eta B_{1,\hat{f_w}} \hat{f_u}(b) + \eta B_{1, \hat{f_u}} \hat{f_w}(b) \bigg] \log\Gamma \bigg(\frac{b}{q} \bigg),
\end{align*}
which is a $\overline{\mathbb{Q}}$-linear relation among log gamma values. Hence, by Conjecture \ref{conj}, we have
\begin{equation*}
    B_{1,\hat{f_l}} \hat{f_k} - B_{1,\hat{f_k}} \hat{f_l} - \eta B_{1,\hat{f_w}} \hat{f_u} + \eta B_{1, \hat{f_u}} \hat{f_w} = 0
\end{equation*}
on all natural numbers. Since $B_{1,\hat{f_j}}$ are non-zero algebraic numbers, we obtain a non-trivial $\overline{\mathbb{Q}}$-linear relation among $\hat{f}_k , \hat{f_l}, \hat{f_u}$ and $\hat{f_w}$. The fact \eqref{lin-ind-f-hat} transports this to $\overline{\mathbb{Q}}$-linear dependence of $f_k$, $f_l$, $f_u$ and $f_w$, which contradicts our hypothesis. Thus, at most one of the $d_{k,l}$'s can be algebraic for $1 \leq k < l \leq r$.

As a result, if four numbers, namely, $L'(1,f_k), L'(1,f_l), L'(1,f_u)$ and $L'(1,f_w)$ are algebraic for $(k,l) \neq (u,w)$, then $d_{k,l}/d_{u,w}$ would be algebraic leading to a contradiction. Hence, the theorem follows.
$\QED$

\subsection{Proof of Theorem \ref{main-theorem-2}}
Using the hypothesis that $L(1,f) = 0$ for all $f \in \mathfrak{G_q}$ and \eqref{L(1,f)-odd}, we know that
\begin{equation*}
    B_{1,\hat{f_j}} = 0,
\end{equation*}
for all $1 \leq j \leq r$. Hence, Lemma \ref{lemma-3} gives
\begin{equation*}
    L'(1,f_j) = \frac{i \pi}{q} \bigg\{ \sum_{b=1}^q \hat{f_j}(b) \log\Gamma \bigg( \frac{b}{q} \bigg) \bigg\},
\end{equation*}
for all $1 \leq j \leq r$. Suppose that for $1 \leq k < l \leq r$, 
\begin{equation*}
    \frac{L'(1,f_k)}{L'(1,f_l)} = \xi \in \overline{\mathbb{Q}}.
\end{equation*}
Then simplifying the above expression gives
\begin{equation*}
    \sum_{b=1}^q \bigg[ \hat{f_k}(b) - \xi \hat{f_l}(b) \bigg] \log\Gamma\bigg( \frac{b}{q}\bigg) = 0,
\end{equation*}
which is an algebraic linear relation among the log gamma values. Therefore, by Conjecture \ref{conj}, we get that
\begin{equation*}
    \hat{f_k} - \xi \hat{f_l} = 0
\end{equation*}
on all natural numbers. This implies the $\overline{\mathbb{Q}}$-linear dependence of the functions $\hat{f_k}$ and $\hat{f_l}$ and thus, contradicts the $\overline{\mathbb{Q}}$-linear independence of $f_k$ and $f_l$ by \eqref{lin-ind-f-hat}. Hence, the quotient $L'(1,f_k)/L'(1,f_l)$ is transcendental for all $1 \leq k < l \leq r$, which in turn leads us to conclude that at most one of the numbers under consideration is algebraic. 
$\QED$

\subsection{Proof of Corollary \ref{coro-1}}

Let $q$ be any natural number greater than $7$ and $\chi$ be an odd primitive Dirichlet character modulo $q$. It suffices to show that $\chi \in \mathfrak{F_q}$ i.e, that $\hat{\chi}$ is of Dirichlet type and that $L(1,\chi) \neq 0$. The latter follows from the famous theorem of Dirichlet \cite[Theorem 6.20 and Section 7.3]{apostol}. The former is also from \cite[Chapter 8, Theorem 8.19]{apostol} since
\begin{equation*}
    \hat{\chi}(n) = \frac{1}{q} \sum_{a=1}^q \chi(a) \zeta_q^{-an}.
\end{equation*}
This completes the proof of the corollary.
$\QED$

\subsection{Proof of Corollary \ref{coro-2}}
We begin the proof by observing that the functions $f_j$ defined below are in $\mathfrak{F_p}$. For $1 \leq j \leq (p-1)/2$,
\begin{equation*}
     f_j(n) :=
    \begin{cases}
    1 & \text{ if } n \equiv j \bmod p, \\
    -1 & \text{ if } n \equiv -j \bmod p, \\
    0 & \text{ otherwise}.
    \end{cases}
\end{equation*}

Clearly, each $f_j$ is periodic with odd prime period $p$, $f_j$ is odd and $f_j(p) = 0$. Moreover, $\sum_{a=1}^p f_j(a) = 0$ and by Theorem \ref{bbw-thm}, $L(1,f_j) \neq 0$ for all $1 \leq j \leq r$. Thus, $f_j \in \mathfrak{F_p}$ for all $1 \leq j \leq r$. Also note that the functions $\{ f_j : 1 \leq j \leq (p-1)/2\}$ are $\overline{\mathbb{Q}}$-linearly independent. Therefore, Theorem \ref{main-theorem-1} implies that at least $\frac{(p-1)}{2} - 3$ of the numbers
\begin{equation*}
    \big\{ \gamma_1(a,p) - \gamma_1(p-a,p) : 1 \leq a \leq (p-1)/2 \big\}
\end{equation*}
are transcendental. Since the difference of two numbers being transcendental implies that at least one of them is transcendental, the result follows. $\QED$

\section{\bf Concluding Remarks}
\bigskip

Our work here represents a modest beginning into the arithmetic nature of generalized Stieltjes constants. These constants have emerged in other contexts. Most notably, they appear in Li's criterion for the Riemann hypothesis (see for example \cite{coffey}). It is quite possible that the study of these
constants can lead us to the holy grail of mathematics.

\section*{Acknowledgments}
We thank the referee and P. Rath for valuable comments on an earlier version of this paper.

\end{document}